\documentclass[12pt,a4paper]{article}
\usepackage[utf8]{inputenc}
\usepackage{tikz-network}
\usepackage{pgfplots}
\usepackage{gensymb}
\usepackage{amsmath}
\usepackage{amssymb}
\usepackage{mathtools}
\usepackage{amsthm}
\usepackage{authblk}
\usepackage{graphics}
\usepackage{enumitem}
\usepackage{graphicx}
\usepackage{pifont}
\usepackage[export]{adjustbox}
\usepackage{geometry}
\usepackage{hyperref}
\newcounter{thm}

\newtheorem{therm}[thm]{Theorem}

\newtheorem{lem}[thm]{Lemma}
\newtheorem{cor}[thm]{Corollary}

\usetikzlibrary{shapes,arrows,positioning}
\title{Interlacing Properties of Eigenvalues of Laplacian and Net-Laplacian Matrix of Signed Graphs}
\author[1]{Satyam Guragain\thanks{Email: shatym17@gmail.com}}
\author[2]{Ravi Srivastava\thanks{Corresponding author, Email: ravi@nitsikkim.ac.in}}
\affil[1,2]{Department of Mathematics, National Institute of Technology, Sikkim 737139, India}
\date{}

\pgfplotsset{compat=1.18}
\begin{document}
\maketitle
\begin{abstract}
This paper explores interlacing inequalities in the Laplacian spectrum of signed cycles and investigates interlacing relationship between the spectrum of the net-Laplacian of a signed graph and its subgraph formed by removing a vertex together with its incident edges. Additionally, an inequality is derived between the net-Laplacian spectrum of a complete co-regular signed graph $\Gamma$ and the Laplacian spectrum of the graph obtained by removing any vertex $v$ from $\Gamma$. Also for a signed graph $\Gamma$, the net-Laplacian matrix is normalized and an inequality is derived between the spectrum of the normalized net-Laplacian of a signed graph and its subgraph, formed by contraction of edge and vertex.
\end{abstract}

\section{Introduction}

All graphs considered in this paper are undirected, simple and finite. A graph $G$ is defined as an ordered pair $G=(V,E)$ where $V$ represents the set of vertices and $E$ represents the set of edges. In 1953 Frank Harary made significant contributions by introducing the idea of signed graphs and the concept of balancedness within such graphs~\cite{ref6}. A signed graph $\Gamma=(G,\sigma)$ consists of a simple graph $G=(V,E)$ and a mapping $\sigma: E\rightarrow \{+,-\}$ known as the signature of $\Gamma$ which assign either positive or negative sign to edges.

The signed degree of a vertex $u$ (denoted by $sdeg(u)$) is the subtraction of the negative degree $(d_v^-)$ from the positive degree $(d_v^+)$ while the degree of the same vertex is simply the sum of its positive and negative degrees. For a signed graph $\Gamma=(G,\sigma)$ with vertex set $V(\Gamma)=\{u_1,u_2,\cdots,u_m\}$ the adjacency matrix of $\Gamma$ is the $m \times m$ matrix given by $A(\Gamma)=(a_{lk}^\sigma)$ where $a_{lk}^\sigma=\sigma(v_lv_k)a_{lk}$ and $a_{lk}=1$ if $v_l$ and $v_k$ are adjacent and $a_{lk}=0$ otherwise.
The Laplacian matrix and the net-Laplacian matrix of $\Gamma$ is given by $L(\Gamma)=D(\Gamma)-A(\Gamma)$ and $N(\Gamma)=D^\pm(\Gamma)-A(\Gamma)$ respectively, where $D(\Gamma)$ is the diagonal matrix with diagonal entries as vertex degree of $\Gamma$ and $D^\pm(\Gamma)=$ diag $(sdeg(u_1),sdeg(u_2),\cdots,$ $sdeg(u_m))$ is the net diagonal matrix of $\Gamma$. Clearly, $L(\Gamma)-N(\Gamma)=2D^-(\Gamma)=2$ diag $(d^-_{u_1},d^-_{u_2},\cdots,d^-_{u_m})$. The minimum negative vertex degree of $\Gamma$ is given by $\delta^{-}(\Gamma)= \text{min}\{d^{-}_u; u\in V(G)\}$ and the maximum negative vertex degree of $\Gamma$ by $\Delta^{-}(\Gamma)= \text{max}\{d^{-}_u; u\in V(G)\}$. If the signed degree of all the vertices in $\Gamma=(G,\sigma)$ is equal to a specific integer $s$, the graph is referred to as net-regular with a net degree of $s$~\cite{ref10}. Also, when the underlying graph $G$ is regular for some integer $r$, and $\Gamma$ is net-regular with a net degree of $s$, the signed graph $\Gamma$ is termed co-regular~\cite{ref11}. The pair $(r,s)$ is the co-regularity pair of $\Gamma$. Furthermore, if $r=|V(\Gamma)|-1$, then $\Gamma$ is a complete co-regular signed graph.

In any signed graph we can switch the sign of edges by a switching function $\theta:V \rightarrow \{+,-\}$~\cite{ref2,ref8}. Switching $\Gamma=(G,\sigma)$ by $\alpha$ create a new signed graph $\Gamma^\alpha=(G,\sigma^\alpha)$ whose underlying graph remains unchanged, but the signature is altered for each edge $e=u_lu_k$ by $\sigma^\alpha(e)=\alpha(u_l)\sigma(e)\alpha(u_k)$. Two signed graph  $\Gamma=(G,\sigma_1)$ and $\Sigma=(G,\sigma_2)$ with same underlying graph $G$ are switching equivalent (denoted as $\Gamma \sim \Sigma$) if there exists a switching function $\alpha$ such that $\sigma_2(e)=\sigma^\alpha_1(e)$ for every edge $e$ in $G$.

Lotker in~\cite{ref9} investigated the impact of removing a vertex on the Laplacian spectrum in an unsigned graph. Later, Wu et al. extended the analysis in~\cite{ref12} by exploring the interlacing relationship between the Laplacian spectra of the original unsigned graph and a graph obtained by deleting more than one vertices together with their incident edges. Also in~\cite{ref5}, Grone et al. demonstrated that in the case of an unsigned graph $G$, the eigenvalues of its Laplacian matrix $L(G)$ exhibit a precise interlacing pattern with the eigenvalues of $L(H)$ where $H$ is a graph obtained by removing a pendant vertex along with its incident edge from $G$. Signed graph also satisfy exactly the same interlacing relation obtained by Lotker in~\cite{ref9} but in section~\ref{sec2} we proved it using different method and use it to obtain an interlacing relation between the Laplacian spectra of two signed cycles whose order differ by $1$. Section~\ref{sec3} comprises of interlacing relation between the net-Laplacian of signed graphs whereas in Section~\ref{sec4} a net-Laplacian matrix is normalised and the interlacing relation between the normalized net-Laplacian spectra of the signed graph $\Gamma$ and its subgraph obtained by contraction (vertex and edge)~\cite{ref1} of $\Gamma$ is discussed.\\ \\
\textbf{Notations and results used:}
\begin{enumerate}[label=\Roman*., itemsep=0.5em]
\item Ordered spectrum $(\beta_1,\beta_2,\cdots,\beta_n)$ refers to eigenvalues arranged in non-decreasing order $i.e$, $\beta_1\leqslant \beta_2 \leqslant \cdots \leqslant \beta_n$
    \item $e_l$ denotes a column vector with 1 in the $l^{th}$ entry and zero elsewhere.
    \item $x \perp W$ implies $x^Ty=0$ for every $y \in W$.
    \item $\sum\limits_{\substack{\Gamma(u_l \sim u_t)}}$ indicates summation over all the unordered vertices $\{u_l,u_t\}$ in the graph $\Gamma$ such that $u_l$ and $u_t$ are adjacent.
    \item \textbf{Courant Fisher Theorem:}\label{v} Let $A\in \mathbb{R}^{m\times m}$ be a symmetric matrix with ordered spectrum $(\beta_1, \beta_2,\cdots, \beta_m)$. Let $p$ be any integer such that $1 \leqslant p \leqslant m$ then\\
\begin{equation*}
\begin{split}
\beta_{p}&=\min_{W_{m-p}} \max_{{\gamma (\neq 0)\in \mathbb{R}^{m}} \atop{\gamma \perp W_{m-p}}}\frac{\gamma^TA\gamma}{\gamma^T\gamma}
\end{split}
\end{equation*}
and
\begin{equation*}
\begin{split}
\beta_{p}&=\max_{W_{p-1}} \min_{{\gamma (\neq 0)\in \mathbb{R}^{m}}\atop{\gamma\perp W_{p-1}}}\frac{\gamma^TA\gamma}{\gamma^T\gamma}
\end{split}
\end{equation*}
where $W_l$ is $l$ dimensional subspace of $\mathbb{R}^m$ for $1\leqslant l \leqslant m$.
\item \textbf{Weyl's Theorem:}\label{vi} Let $\Lambda$ and $\Theta$ be two $m\times m$ complex Hermitian matrices and the respective ordered spectrum of $\Lambda$, $\Theta$ and $\Lambda+\Theta$ are $(\beta_1(\Lambda),\beta_2(\Lambda),\cdots,\beta_m(\Lambda))$,\\ $(\beta_1(\Theta),\beta_2(\Theta),\cdots,\beta_m(\Theta))$ and $(\beta_1(\Lambda+\Theta),\beta_2(\Lambda+\Theta),\cdots,\beta_m(\Lambda+\Theta))$ then for each $p=1,2,\cdots,m$ we have
\[
\beta_p(\Lambda)+\beta_1(\Theta) \leqslant \beta_p(\Lambda+\Theta) \leqslant \beta_p(\Lambda)+\beta_m(\Theta)
\]
\end{enumerate}
The proof of Courant-Fisher theorem and Wely's theorem is available in [chapter 4,~\cite{ref7}]
\section{Interlacing properties of eigenvalues of Laplacian Matrix}
\label{sec2}
\begin{therm}\label{Theorem 2.1}{\textnormal {Suppose $\Gamma=(G,\sigma)$ is a signed graph with vertex set $V(\Gamma)=\{u_1,u_2,\cdots,u_{m+1}\}$ and ordered Laplacian spectrum $(\alpha_1,\alpha_2,\cdots, \alpha_{m+1})$ and $\Gamma'$ is a signed graph obtained by removing any one vertex $v$ of $\Gamma$. Let $(\beta_1, \beta_2,\cdots, \beta_m)$ be the ordered Laplacian spectrum of $\Gamma'$ then\\
\[
\alpha_p \leqslant \beta_p+1 \leqslant \alpha_{p+1}+1~~\text{for}~~ p=1,2,....,m
\]}}
\end{therm}

\begin{proof}
Suppose $v=u_{m+1}$ and $p$ be any integer such that $1 \leqslant p \leqslant m$.
\begin{math}
 \text{Let} ~\hat{\gamma} = \begin{bmatrix}
           \gamma \\
           \gamma_{m+1} \\
         \end{bmatrix}
\end{math}
$\in \mathbb{R}^{m+1}, \gamma \in \mathbb{R}^m, \gamma_{m+1} \in \mathbb{R}$ and
$W_s$ and $\hat{W}_s$ be $s$-dimensional subspace of $\mathbb{R}^m$ and $\mathbb{R}^{m+1}$ respectively.
Using Courant-Fisher Theorem \ref{v},\\
\begin{equation*}
\begin{split}
\alpha_{p+1}&=\min_{\hat{W}_{(m+1)-(p+1)}}~ \max\limits_{\substack{\hat{\gamma} (\ne 0) \in \mathbb{R}^{m+1}}\atop{\hat{\gamma}\perp \hat{W}_{(m+1)-(p+1)}}}~\frac{\hat{\gamma}^TL(\Gamma)\hat{\gamma}}{\hat{\gamma}^T\hat{\gamma}}\\
&=\min_{\hat{W}_{m-p}} ~\max\limits_{\substack{\hat{\gamma} (\ne 0) \in \mathbb{R}^{m+1}}\atop{\hat{\gamma}\perp \hat{W}_{m-p}}}~\frac{\sum\limits_{\substack{\Gamma(u_i \sim u_j)}}(\gamma_i-\sigma(u_iu_j)\gamma_j)^2}{\gamma_1^2+\gamma_2^2+\cdots+\gamma_m^2+\gamma_{m+1}^2}  \\ \\
  &\geqslant \min_{\hat{W}_{m-p}} ~\max\limits_{\substack{\hat{\gamma} (\ne 0) \in \mathbb{R}^{m+1}\\\hat{\gamma}\perp \hat{W}_{m-p}}\atop{\hat{\gamma} \perp e_{m+1}}}~\frac{\sum\limits_{\substack{\Gamma'(u_i \sim u_j)}}(\gamma_i-\sigma(u_iu_j)\gamma_j)^2+\sum\limits_{\substack{\Gamma(u_{n+1} \sim u_j)}}(\gamma_{m+1}-\sigma(u_{m+1}u_j)\gamma_j)^2}{\gamma_1^2+\gamma_2^2+\cdots+\gamma_m^2+\gamma_{m+1}^2}
\end{split}
\end{equation*}\\
As $\hat{\gamma} \perp e_{m+1}$ we can consider $\gamma_{m+1}=0$ and simply delete $(m+1)^{th}$ index of elements of $\hat{W}_{m-p}$.\\ \\
\begin{equation*}
    \begin{split}
\therefore ~\alpha_{p+1}&\geqslant \min_{W_{m-p}}~ \max\limits_{\substack{\gamma (\ne 0) \in \mathbb{R}^{m}\\\atop{\gamma\perp W_{m-p}}}}~\frac{\sum\limits_{\substack{\Gamma'(u_i \sim u_j)}}(\gamma_i-\sigma(u_iu_j)\gamma_j)^2+\sum\limits_{\substack{\Gamma(u_{m+1} \sim u_j)}}\gamma_j^2}{\gamma_1^2+\gamma_2^2+\cdots+\gamma_m^2}  \\ \\
&\geqslant \min_{W_{m-p}}~ \max\limits_{\substack{\gamma (\ne 0) \in \mathbb{R}^{m}\\\atop{\gamma \perp W_{m-p}}}}~\frac{\sum\limits_{\substack{\Gamma'(u_i \sim u_j)}}(\gamma_i-\sigma(u_iu_j)\gamma_j)^2}{\gamma_1^2+\gamma_2^2+\cdots+\gamma_m^2}=\beta_p
\end{split}
\end{equation*}
and
\begin{equation*}
\begin{split}
\alpha_{p}&=\max_{\hat{W}_{p-1}}~ \min\limits_{\substack{\hat{\gamma} (\ne 0) \in \mathbb{R}^{m+1}}\atop{\hat{\gamma}\perp \hat{W}_{p-1}}}~\frac{\hat{\gamma}^TL(\Gamma)\hat{\gamma}}{\hat{\gamma}^T\hat{\gamma}}  \\
&=\max_{\hat{W}_{p-1}}~ \min\limits_{\substack{\hat{\gamma} (\ne 0) \in \mathbb{R}^{m+1}}\atop{\hat{\gamma}\perp \hat{W}_{p-1}}}~\frac{\sum\limits_{\substack{\Gamma(u_i \sim u_j)}}(\gamma_i-\sigma(u_iu_j)\gamma_j)^2}{\gamma_1^2+\gamma_2^2+\cdots+\gamma_n^2+\gamma_{m+1}^2}\\
&\leqslant \max_{\hat{W}_{p-1}}~ \min\limits_{\substack{\hat{\gamma} (\ne 0) \in \mathbb{R}^{m+1}\\\hat{\gamma}\perp \hat{W}_{p-1}}\atop{\hat{\gamma} \perp e_{m+1}}}~\frac{\sum\limits_{\substack{\Gamma'(u_i \sim u_j)}}(\gamma_i-\sigma(u_iu_j)\gamma_j)^2+\sum\limits_{\substack{\Gamma(u_{m +1} \sim u_j)}}(\gamma_{m+1}-\sigma(u_{m+1}u_j)\gamma_j)^2}{\gamma_1^2+\gamma_2^2+\cdots+\gamma_m^2+\gamma_{m+1}^2}  \\
&= \max_{W_{p-1}}~ \min\limits_{\substack{\gamma (\ne 0) \in \mathbb{R}^{m}\\\atop{\gamma\perp W_{p-1}}}}~\frac{\sum\limits_{\substack{\Gamma'(u_i \sim u_j)}}(\gamma_i-\sigma(u_iu_j)\gamma_j)^2+\sum\limits_{\substack{\Gamma(u_{m+1} \sim u_j)}}\gamma_j^2}{\gamma_1^2+\gamma_2^2+\cdots+\gamma_m^2}  \\
&\leqslant \max_{W_{p-1}}~ \min\limits_{\substack{\gamma (\ne 0) \in \mathbb{R}^{m}\\\atop{\gamma \perp W_{p-1}}}}~\Bigg[\frac{\sum\limits_{\substack{\Gamma'(u_i \sim u_j)}}(\gamma_i-\sigma(u_iu_j)\gamma_j)^2}{\gamma_1^2+\gamma_2^2+\cdots+\gamma_m^2}+1\Bigg]=\beta_p+1
\end{split}
\end{equation*}
Thus, $\alpha_p \leq \beta_p+1 \leq \alpha_{p+1}+1$ for $p=1,2,3,\cdots,m$.
\end{proof}
\begin{cor}{\textnormal{Suppose $\Gamma=(G,\sigma)$ is a signed graph of order $m+1$ whose vertex $u$ is adjacent to all the remaining vertices of $G$. If $\Gamma'=(G',\sigma)$ is a signed graph obtained by removing vertex $u$ from $\Gamma$ and $(\alpha_1, \alpha_2, \cdots, \alpha_{m+1})$ and $(\beta_1, \beta_2, \cdots, \beta_m)$ are respective ordered Laplacian spectrum of $\Gamma$ and $\Gamma'$ then,\\
\[
\alpha_p \leqslant \beta_p +1 \leqslant \alpha_{p+1}~~\text{for}~~ p=1,2,....,m
\]}}
\end{cor}
\begin{proof} Let $u=u_{m+1}$. So,  ${\sum\limits_{\substack{\Gamma(u_{m+1} \sim u_j)}}(\gamma_{m+1}-\sigma(u_{m+1}u_j)\gamma_j)^2}={\displaystyle \sum_{j=1}^{m}(\gamma_{m+1}-\sigma(u_{m+1}u_j)\gamma_j)^2}$. Thus
\begin{equation*}
\begin{split}
\alpha_{p+1}&\geqslant \min_{W_{m-p}}~ \max\limits_{\substack{\gamma (\neq 0) \in \mathbb{R}^{m}\\\atop{\gamma\perp W_{m-p}}}}~\frac{\sum\limits_{\substack{\Gamma'(u_i \sim u_j)}}(\gamma_i-\sigma(u_iu_j)\gamma_j)^2+\displaystyle \sum_{j=1}^{m}\gamma_j^2}{\gamma_1^2+\gamma_2^2+\cdots+\gamma_m^2}  \\ \\
&= \min_{W_{m-p}} ~\max\limits_{\substack{\gamma (\ne 0) \in \mathbb{R}^{m}\\\atop{\gamma \perp W_{m-p}}}}~\Bigg[\frac{\sum\limits_{\substack{\Gamma'(u_i \sim u_j)}}(\gamma_i-\sigma(u_iu_j)\gamma_j)^2}{\gamma_1^2+\gamma_2^2+\cdots+\gamma_m^2}+1\Bigg]=\beta_p+1
\end{split}
\end{equation*}
and
\begin{equation*}
\begin{split}
\alpha_p&\leqslant \max_{W_{p-1}}~ \min\limits_{\substack{\gamma (\ne 0) \in \mathbb{R}^{m}\\\atop{\gamma\perp W_{p-1}}}}~\frac{\sum\limits_{\substack{\Gamma'(u_i \sim u_j)}}(\gamma_i-\sigma(u_iu_j)\gamma_j)^2+\displaystyle \sum_{j=1}^{m}\gamma_j^2}{\gamma_1^2+\gamma_2^2+\cdots+\gamma_m^2}  \\ \\
&= \max_{W_{p-1}} ~\min\limits_{\substack{\gamma (\ne 0) \in \mathbb{R}^{m}\\\atop{\gamma \perp W_{p-1}}}}~\Bigg[\frac{\sum\limits_{\substack{\Gamma'(u_i \sim u_j)}}(\gamma_i-\sigma(u_iu_j)\gamma_j)^2}{\gamma_1^2+\gamma_2^2+\cdots+\gamma_m^2}+1\Bigg]=\beta_p+1
\end{split}
\end{equation*}
Thus, $\alpha_p \leqslant \beta_p+1 \leqslant \alpha_{p+1}$ for $p=1,2,3,\cdots,m$.
\end{proof}
\begin{lem}\label{lemma 2.3}{\textnormal{Let $\Gamma=(G,\sigma)$ be a signed graph with $m$ vertices and Laplacian spectrum $(\alpha_1, \alpha_2, \cdots \alpha_{m+1})$ and $\Gamma'$ is the signed graph obtained by removing any one edge $uv \in E(\Gamma)$ whose ordered Laplacian spectrum is $(\beta_1, \beta_2, \cdots, \beta_m)$ then,
\[
\beta_p \leqslant \alpha_p \leqslant \beta_p+2 ~~\text{for}~~ k=1,2,....,m
\]}}
\end{lem}
\begin{proof} Follows directly from Weyl's Theorem \ref{vi}.
\end{proof}
\begin{therm}\label{Theorem 2.4}{\textnormal{Suppose $(C_{m+1},\sigma_1)$ be a signed cycle with vertex set $\{v_1,v_2,\cdots,v_{m+1}\}$ and ordered Laplacian spectrum $(\alpha_1, \alpha_2,\cdots, \alpha_{m+1})$ and $(C_{m},\sigma_2)$ is a signed cycle with vertex set $\{v_1,v_2,\cdots,v_m\}$ and ordered Laplacian spectrum $(\beta_1, \beta_2,\cdots, \beta_m)$ with $\sigma_2(e)=\sigma_1(e)$ for all $e \in E(C_m)$ except for the edge $v_1v_m$ and $\sigma_2(v_1v_m)$ is defined randomly then,
\[
\alpha_p-1 \leqslant \beta_p \leqslant \alpha_{p+1}+2 ~~\text{for}~~p=1,2,....,m
\]}}
\end{therm}
\begin{proof} Suppose $(H,\sigma_1)$ be a signed graph obtained by removing vertex $v_{m+1}$ from $(C_{m+1},\sigma_1)$ and $(\epsilon_1, \epsilon_2, \cdots, \epsilon_m)$ be the ordered Laplacian eigenvalues of $(H,\sigma_1)$ then by Theorem \ref{Theorem 2.1}\\
\begin{equation}
\label{equation 2.1}
\alpha_p \leqslant \epsilon_p+1 \leqslant \alpha_{p+1}+1~~~~ \text{for}~~~~~ p=1, 2, \cdots, m
\end{equation}
Also, $(C_{m},\sigma_2)$ is obtained by joining edge $u_1u_m$ in $(H,\sigma_1)$ and randomly assigning any sign to $v_1v_m$ then by Lemma \ref{lemma 2.3}
\begin{equation}
\label{equation 2.2}
    \epsilon_p \leqslant \beta_p \leqslant \epsilon_p+2~~\text{for}~~p=1, 2, \cdots, m
\end{equation}
Using equation \ref{equation 2.1} and equation \ref{equation 2.2} we get $\alpha_p-1 \leqslant \beta_p \leqslant \alpha_{p+1}+2~~\text{for}~~p=1,2,....,m$.
\end{proof}
\begin{cor}{\textnormal{Theorem 2.5 holds for any signature function on $C_{m+1}$ and $C_m$.}}
\end{cor}
\begin{proof} Let L=$(C_{m+1},\sigma_1)$ and H=$(C_m,\sigma_2)$, where $\sigma_1$ and $\sigma_2$ are random sign functions on $C_{m+1}$ and $C_m$ respectively. We can switch $L$ accordingly as it is balanced or unbalanced shown in Figure \ref{fig 1}.
\begin{figure}[ht]
    \centering
   \includegraphics[scale=0.5]{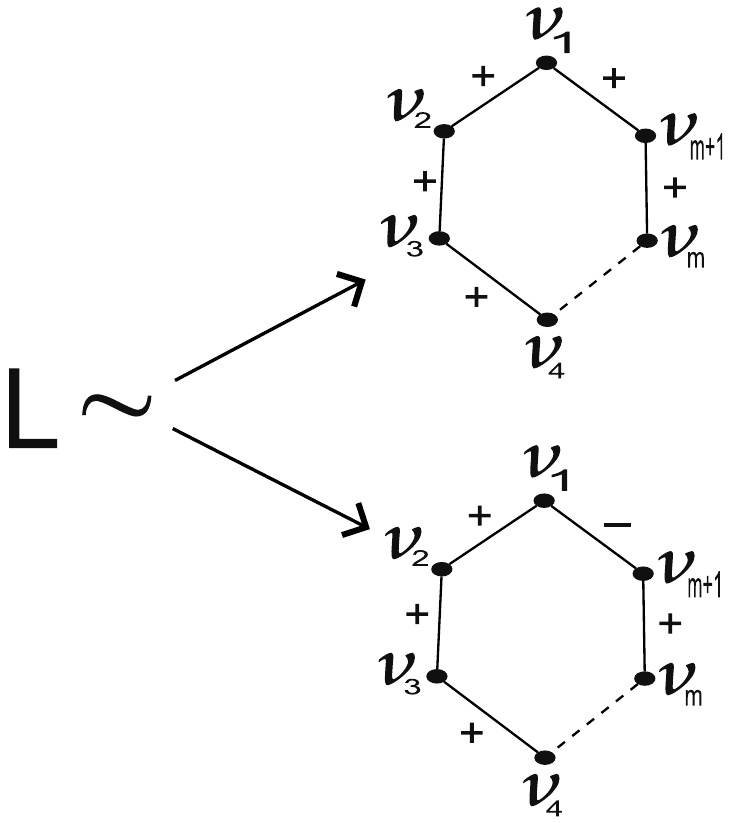}
    \caption{Possible signed cycle of order $m+1$ upto switching isomorphism.}
    \label{fig 1}
\end{figure}\\
If H is balanced then we can remove vertex $v_{m+1}$ from L and join vertices $v_1$ and $v_m$ with $\sigma_1(v_1v_m)=+$ to obtain a signed graph switching equivalent to H and
if H is unbalanced then we can remove $v_{m+1}$ from L and join vertices $v_1$ and $v_m$ with $\sigma_1(v_1v_m)=-$ to obtain a signed graph switching equivalent to H.\\
In both the cases result follows by Theorem \ref{Theorem 2.4}.
\end{proof}
\begin{lem}\label{lemma 2.6}{\textnormal{ \begin{math}
 \text{Let} ~\gamma = \begin{bmatrix}
           \gamma_1 \\
           \gamma_2 \\
           \vdots\\
           \gamma_m
         \end{bmatrix}
\end{math}
$\in \mathbb{R}^m$ and $f(\gamma_1,\gamma_2,\cdots,\gamma_m)$ be a bounded function on $m$ variables then for any $\gamma_{m+1}\in \mathbb{R}$
\begin{equation*}
\begin{split}
    \max_{\gamma\ne 0} f(\gamma_1,\gamma_2,\cdots,\gamma_m) = \max_{\hat{\gamma}\ne 0} f(\gamma_1,\gamma_2,\cdots,\gamma_m)~~\text{and}~~\min_{\gamma\ne 0} f(\gamma_1,\gamma_2,\cdots,\gamma_m) = \min_{\hat{\gamma}\ne 0} f(\gamma_1,\gamma_2,\cdots,\gamma_m)
     \end{split}
 \end{equation*}
where \begin{math}
 \hat{\gamma} = \begin{bmatrix}
           \gamma_1 \\
           \gamma_2 \\
           \vdots\\
           \gamma_m\\
           \gamma_{m+1}
         \end{bmatrix}
 \in \mathbb{R}^{m+1}
\end{math}}}
\end{lem}
 \begin{proof} Let function $f$ attains its maximum value at $\alpha=(r_1,r_2,\cdots,r_m)^T$ and minimum value at $\beta=(t_1,t_2,\cdots,t_m)^T$ then in $(m+1)^{th}$ dimension $\hat{\alpha}=(r_1,r_2,\cdots,r_m,a)^T$ gives the maximum value of $f$ for any $a\in \mathbb{R}$ and $\hat{\beta}=(t_1,t_2,\cdots,t_m,b)^T$ gives the minimum value of $f$ for any $b\in \mathbb{R}$.
 \end{proof}
\begin{therm}{\textnormal{Let $\Sigma=(G,\sigma)$ be a signed graphs of order $m+1$ with $u\in V(\Sigma)$ such that $d(u)=1$ and $\Sigma'=(G',\sigma)$ be a signed path obtained by removing vertex $u$ from $\Sigma$. Suppose $uv\in E(\Gamma)$ with $\sigma(uv)=\star$ where $\star \in \{+,-\}$. If $(\alpha_1, \alpha_2 , \cdots, \alpha_{m+1})$ and $(\beta_1, \beta_2, \cdots, \beta_m)$ are respective ordered Laplacian spectrum of $\Sigma$ and $\Sigma'$ then,\\
\[
\alpha_p \leqslant \beta_p \leqslant \alpha_{p+1}~~\text {for}~~ p=1,2,....,m
\]}}
\end{therm}
\begin{proof} Taking
\begin{math}
  ~\hat{\gamma} = \begin{bmatrix}
           \gamma \\
           \gamma_{m+1} \\
         \end{bmatrix}
\end{math}
$\in \mathbb{R}^{m+1}$ where
\begin{math}
 \gamma=\begin{bmatrix}
            \gamma_1\\
            \gamma_2\\
            \vdots\\
            \gamma_m
 \end{bmatrix}
\end{math}$\in \mathbb{R}^m, \gamma_{m+1} \in \mathbb{R}$ and
$W_l$, $\hat{W}_l$ as $l$-dimensional subspace of $\mathbb{R}^m$, $\mathbb{R}^{m+1}$ respectively and
proceeding as on Theorem \ref{Theorem 2.1} we get $\alpha_{p+1} \geqslant \beta_p$ for $p=1,2,\cdots,m$. Now to prove $\beta_p \geqslant \alpha_p$ for $p=1,2,\cdots,m$ let consider $V(\Sigma)=\{u_1,u_2,\cdots,u_m,u_{m+1}\}$, $u=u_{m+1}$ and $v=u_{m}$ with $u_{m+1}u_m \in E(\Sigma)$. Firstly we consider the case $\star$ is negative then,\\
\begin{equation*}
\begin{split}
\beta_{p}&=\min_{W_{m-p}}~\max\limits_{\substack{\gamma (\ne 0) \in \mathbb{R}^m}\atop{\gamma\perp W_{m-p}}}~\frac{\gamma^TL(\Sigma')\gamma}{\gamma^T\gamma}\\
&=\min_{W_{m-p}}~\max\limits_{\substack{\gamma (\ne 0) \in \mathbb{R}^m}\atop{\gamma\perp W_{m-p}}}~\frac{\sum\limits_{\substack{\Sigma'(u_i \sim u_j)}}(\gamma_i-\sigma(u_iu_j)\gamma_j)^2}{\gamma_1^2+\gamma_2^2+\cdots+\gamma_m^2}\\
&=\min_{\hat{W}_{m-p}}~ \max\limits_{\substack{\hat{\gamma} (\ne 0) \in \mathbb{R}^{m+1}}\atop{\hat{\gamma}\perp \hat{W}_{m-p}}}~\frac{\sum\limits_{\substack{\Sigma(u_i \sim u_j)}}(\gamma_i-\sigma(u_iu_j)\gamma_j)^2-(\gamma_m+\gamma_{m+1})^2}{\gamma_1^2+\gamma_2^2+\cdots+\gamma_m^2}~~~~~~~~~~~~~~[\text{by Lemma \ref{lemma 2.6}]}\\
 &\geqslant \min_{\hat{W}_{m-p}}~\max\limits_{\substack{\hat{\gamma
} (\ne 0) \in \mathbb{R}^{m+1}\\\hat{\gamma}\perp \hat{W}_{m-p}}\atop{\gamma_m=-\gamma_{m+1}}}~\frac{\sum\limits_{\substack{\Sigma(u_i \sim u_j)}}(\gamma_i-\sigma(u_iu_j)\gamma_j)^2-(\gamma_m+\gamma_{m+1})^2}{\gamma_1^2+\gamma_2^2+\cdots+\gamma_m^2}\\
 &= \min_{\hat{W}_{m-p}}~\max\limits_{\substack{\hat{\gamma
} (\ne 0) \in \mathbb{R}^{m+1}\\\hat{\gamma}\perp \hat{W}_{m-p}}\atop{\gamma_m=-\gamma_{m+1}}}~\frac{\sum\limits_{\substack{\Sigma(u_i \sim u_j)}}(\gamma_i-\sigma(u_iu_j)\gamma_j)^2}{\gamma_1^2+\gamma_2^2+\cdots+\gamma_m^2}\\
 &\geqslant \min_{\hat{W}_{m-p}}~\max\limits_{\substack{\hat{\gamma
} (\ne 0) \in \mathbb{R}^{m+1}\\\hat{\gamma}\perp \hat{W}_{m-p}}\atop{\hat{\gamma}\perp e_m+e_{m+1}}}~\frac{\sum\limits_{\substack{\Sigma(u_i \sim u_j)}}(\gamma_i-\sigma(u_iu_j)\gamma_j)^2}{\gamma_1^2+\gamma_2^2+\cdots+\gamma_m^2+\gamma_{m+1}^2}\\
&= \min_{\hat{W}_{m+1-p}}~\max\limits_{\substack{\hat{\gamma
} (\ne 0) \in \mathbb{R}^{m+1}\\ \hat{\gamma}\perp \hat{W}_{m+1-p}}}~\frac{\hat{\gamma}^TL(\Sigma)\hat{\gamma}}{\hat{\gamma}^T\hat{\gamma}}=\alpha_p
\end{split}
\end{equation*}
The result can be proved similarly for $\star$ as positive.
\end{proof}
\begin{cor}{\textnormal{Let $\Gamma=(P_{m+1},\sigma_1)$ and $\Gamma' =(P_{m},\sigma_2)$ be signed paths with $(m+1)$ and $m$ vertices respectively. If $(\alpha_1, \alpha_2, \cdots, \alpha_{m+1})$ and $(\beta_1, \beta_2, \cdots, \beta_m)$ are respective ordered Laplacian spectrum of $\Gamma$ and $\Gamma'$ then,
\[
\alpha_p \leqslant \beta_p \leqslant \alpha_{p+1}~~\text{for}~~p=1,2,....,n
\]}}
\end{cor}
\begin{cor}{\textnormal{Let $\Gamma=(K_{1,m},\sigma_1)$ and $\Gamma' =(K_{1,m-1},\sigma_2)$ be signed stars of order $(m+1)$ and $m$ respectively. If $(\alpha_1, \alpha_2, \cdots, \alpha_{m+1})$ and $(\beta_1, \beta_2, \cdots, \beta_m)$ are respective ordered Laplacian spectrum of $\Gamma$ and $\Gamma'$ then,
\[
\alpha_p \leqslant \beta_p \leqslant \alpha_{p+1}~~\text{for}~~p=1,2,....,n
\]}}
\end{cor}

\section{Interlacing properties of eigenvalues of net-Laplacian of signed graphs.}
\label{sec3}
\begin{lem}\label{lem 3.1}{\textnormal{Suppose $(G,\sigma)$ is a signed graph with vertex set \{$u_1,u_2,\cdots,u_m$\}. If $(\alpha_1,\alpha_2, \cdots, \alpha_{m})$ and $(\beta_1, \beta_2, \cdots,\beta_m)$ are the respective ordered Laplacian spectrum and net-Laplacian spectrum of $(G,\sigma)$ then for $p=1,2,\cdots,m$
\[
\beta_p+2~\delta^{-}(G) \leqslant \alpha_p \leqslant \beta_p+2~\Delta^{-}(G)
\]}}
\end{lem} 
\begin{proof} Follows directly from Weyl's Theorem \ref{vi}.
\end{proof}
\begin{therm}\label{theorem 3.2}{\textnormal{Let $\Gamma=(G,\sigma)$ be a signed graph with $m$ vertices and $\Gamma'=\Gamma-e$ be the signed graph obtained from $\Gamma$ by removing an edge $e$ such that $\sigma(e)=-$. If $(\alpha_1,\alpha_2,\cdots, \alpha_{m})$ and $(\beta_1, \beta_2,\cdots, \beta_m)$ be the ordered net-Laplacian spectrum of $\Gamma$ and $\Gamma'$ respectively, then\\
\[
\alpha_p \leqslant \beta_p \leqslant \alpha_{p+1}~\text{ for }~p=1,2,\cdots,m ~~\text{ with the convention that }~\alpha_{m+1}=m
\]}}
\end{therm}
\begin{proof} Suppose $p$ be any integer such that $1 \leqslant p \leqslant m$.
\begin{math}
 \text{Let} ~{\gamma} = \begin{bmatrix}
           \gamma_1 \\
           \gamma_2 \\
           \vdots\\
           \gamma_m
         \end{bmatrix}
\end{math}
$\in \mathbb{R}^{m}$, ${W}_l$ be $l$-dimensional subspace of $\mathbb{R}^m$.
Using Courant-Fisher Theorem \ref{v},\\
\begin{equation*}
\begin{split}
\beta_{p}&=\min_{{W}_{m-p}}~ \max\limits_{\substack{\gamma (\ne 0) \in \mathbb{R}^{m}}\atop{\gamma\perp W_{m-p}}}~\frac{\gamma^TN(\Gamma')\gamma}{\gamma^T\gamma}\\
&=\min_{W_{m-p}} ~\max\limits_{\substack{\gamma (\ne 0) \in \mathbb{R}^{m}}\atop{\gamma\perp W_{m-p}}}~\frac{\sum\limits_{\substack{\Gamma'(u_i \sim u_j)}}(\gamma_i-\sigma(u_iu_j)\gamma_j)^2-2\sum\limits_{\substack{u_i\in V(\Gamma')}}d_i^-\gamma_i^2}{\gamma_1^2+\gamma_2^2+\cdots+\gamma_m^2}  \\ \\
&=\min_{W_{m-p}} ~\max\limits_{\substack{\gamma (\ne 0) \in \mathbb{R}^{m}}\atop{\gamma\perp W_{m-p}}}~\frac{\sum\limits_{\substack{\Gamma(u_i \sim u_j)}}(\gamma_i-\sigma(u_iu_j)\gamma_j)^2-2\sum\limits_{\substack{u_i\in V(\Gamma)}}d_i^-\gamma_i^2+(\gamma_1-\gamma_2)^2}{\gamma_1^2+\gamma_2^2+\cdots+\gamma_m^2}  \\ \\
&\geqslant \min_{W_{m-p}} ~\max\limits_{\substack{\gamma (\ne 0) \in \mathbb{R}^{m}}\atop{\gamma\perp W_{m-p}}}~\frac{\sum\limits_{\substack{\Gamma(u_i \sim u_j)}}(\gamma_i-\sigma(u_iu_j)\gamma_j)^2-2\sum\limits_{\substack{u_i\in V(\Gamma)}}d_i^-\gamma_i^2}{\gamma_1^2+\gamma_2^2+\cdots+\gamma_m^2}=\alpha_p
\end{split}
\end{equation*}
and
\begin{equation*}
\begin{split}
\beta_{p}&=\max_{{W}_{p-1}}~ \min\limits_{\substack{\gamma (\ne 0) \in \mathbb{R}^{m}}\atop{\gamma\perp W_{p-1}}}~\frac{\gamma^TN(\Gamma')\gamma}{\gamma^T\gamma}\\
&=\max_{W_{p-1}} ~\min\limits_{\substack{\gamma (\ne 0) \in \mathbb{R}^{m}}\atop{\gamma\perp W_{p-1}}}~\frac{\sum\limits_{\substack{\Gamma'(u_i \sim u_j)}}(\gamma_i-\sigma(u_iu_j)\gamma_j)^2-2\sum\limits_{\substack{u_i\in V(\Gamma')}}d_i^-\gamma_i^2}{\gamma_1^2+\gamma_2^2+\cdots+\gamma_m^2}  \\ \\
&=\max_{W_{p-1}} ~\min\limits_{\substack{\gamma (\ne 0) \in \mathbb{R}^{m}}\atop{\gamma\perp W_{p-1}}}~\frac{\sum\limits_{\substack{\Gamma(u_i \sim u_j)}}(\gamma_i-\sigma(u_iu_j)\gamma_j)^2-2\sum\limits_{\substack{u_i\in V(\Gamma)}}d_i^-\gamma_i^2+(\gamma_1-\gamma_2)^2}{\gamma_1^2+\gamma_2^2+\cdots+\gamma_m^2}
\end{split}
\end{equation*}
\begin{equation*}
\begin{split}
\beta_p&\leqslant \max_{W_{p-1}}~ \min\limits_{\substack{\gamma (\ne 0) \in \mathbb{R}^{m}\\\gamma\perp W_{p-1}}\atop{\gamma \perp e_1-e_2}}~\frac{\sum\limits_{\substack{\Gamma(u_i \sim u_j)}}(\gamma_i-\sigma(u_iu_j)\gamma_j)^2-2\sum\limits_{\substack{u_i\in V(\Gamma)}}d_i^-\gamma_i^2+(\gamma_1-\gamma_2)^2}{\gamma_1^2+\gamma_2^2+\cdots+\gamma_m^2}\\
&=\max_{W_{p}} ~\min\limits_{\substack{\gamma (\ne 0) \in \mathbb{R}^{m}}\atop{\gamma\perp W_{p}}}~\frac{\sum\limits_{\substack{\Gamma(u_i \sim u_j)}}(\gamma_i-\sigma(u_iu_j)\gamma_j)^2-2\sum\limits_{\substack{u_i\in V(\Gamma)}}d_i^-\gamma_i^2}{\gamma_1^2+\gamma_2^2+\cdots+\gamma_m^2}=\alpha_{p+1} 
\end{split}
\end{equation*}
Thus, $\alpha_p \leqslant \beta_p \leqslant \alpha_{p+1}$ for $p=1,2,\cdots,m$
\end{proof}
\begin{therm}{\textnormal{Let $\Gamma=(G,\sigma)$ be a signed graph with $m$ vertices and $\Gamma'=\Gamma-e$ be the signed graph obtained from $\Gamma$ by removing an edge $e$ such that $\sigma(e)=+$. If $(\alpha_1,\alpha_2,\cdots, \alpha_{m})$ and $(\beta_1, \beta_2,\cdots, \beta_m)$ be the ordered net-Laplacian eigenvalues of $\Gamma$ and $\Gamma'$ respectively, then\\
\[
\alpha_{p-1} \leqslant \beta_{p} \leqslant \alpha_{p}~\text{ for }~p=1,2,\cdots,m~\text{ with the convention that } \alpha_0=-m
\]}}
\end{therm}
\begin{proof} Similar to that of Theorem \ref{theorem 3.2}.
\end{proof}
\begin{therm}\label{thm 3.4}{\textnormal{Let $\Gamma=(G,\sigma)$ be connected signed graph with $(m+1)$ vertices and ordered net-Laplacian eigenvalues $(\alpha_1,\alpha_2,\cdots, \alpha_{m+1})$ and $\Gamma'=(G',\sigma)$ is a signed graph obtained by removing any one vertex of $\Gamma$ and $(\beta_1, \beta_2,\cdots, \beta_m)$ be the ordered net-Laplacian eigenvalues of $\Gamma'$ then\\
\[
\alpha_p-1 \leqslant \beta_p \leqslant \alpha_{p+1}+1~~\text{for}~~ p=1,2,\cdots,m
\]}}
\end{therm}
\begin{proof} Suppose $p$ be any integer such that $1 \leqslant p \leqslant m$.
\begin{math}
 \text{Let} ~\hat{\gamma} = \begin{bmatrix}
           \gamma \\
           \gamma_{m+1} \\
         \end{bmatrix}
\end{math}
$\in \mathbb{R}^{m+1}, \gamma\in \mathbb{R}^m, \gamma_{m+1} \in \mathbb{R}$ and
$W_l$ and $\hat{W}_l$ be $l$-dimensional subspace of $\mathbb{R}^m$ and $\mathbb{R}^{m+1}$ respectively.
Using Courant-Fisher Theorem \ref{v},\\
\begin{equation*}
\begin{split}
\alpha_{p+1}&=\min_{\hat{W}_{(m+1)-(p+1)}}~ \max\limits_{\substack{\hat{\gamma} (\ne 0) \in \mathbb{R}^{m+1}}\atop{\hat{\gamma}\perp \hat{W}_{(m+1)-(p+1)}}}~\frac{\hat{\gamma}^TN(\Gamma)\hat{\gamma}}{\hat{\gamma}^T\hat{\gamma}}\\
&=\min_{\hat{W}_{m-p}} ~\max\limits_{\substack{\hat{\gamma} (\ne 0) \in \mathbb{R}^{m+1}}\atop{\hat{\gamma}\perp \hat{W}_{m-p}}}~\frac{\sum\limits_{\substack{\Gamma(u_i \sim u_j)}}(\gamma_i-\sigma(u_iu_j)\gamma_j)^2-2\sum\limits_{\substack{u_i\in V(\Gamma)}}d_i^-\gamma_i^2}{\gamma_1^2+\gamma_2^2+\cdots+\gamma_m^2+\gamma_{m+1}^2}  \\ \\
  &\geqslant \min_{\hat{W}_{m-p}} ~\max\limits_{\substack{\hat{\gamma} (\ne 0) \in \mathbb{R}^{m+1}\\\hat{\gamma}\perp \hat{W}_{m-p}}\atop{\hat{\gamma} \perp e_{m+1}}}~\frac{\sum\limits_{\substack{\Gamma'(u_i \sim u_j)}}(\gamma_i-\sigma(u_iu_j)\gamma_j)^2+\sum\limits_{\substack{\Gamma(u_{m+1} \sim u_j)}}(\gamma_{m+1}-\sigma(u_{m+1}u_j)\gamma_j)^2-2\sum\limits_{\substack{u_i\in V(\Gamma)}}d_i^-\gamma_i^2}{\gamma_1^2+\gamma_2^2+\cdots+\gamma_m^2+\gamma_{m+1}^2}
\end{split}
\end{equation*}
As $\hat{\gamma} \perp e_{m+1}$ we can consider $\gamma_{m+1}=0$ and simply delete $(m+1)^{th}$ index of elements of $\hat{W}_{m-p}$.\\
\begin{equation*}
    \begin{split}
\therefore ~\alpha_{p+1}&\geqslant \min_{W_{m-p}}~ \max\limits_{\substack{\gamma (\ne 0) \in \mathbb{R}^{m}\\\atop{\gamma\perp W_{m-p}}}}~\frac{\sum\limits_{\substack{\Gamma'(u_i \sim u_j)}}(\gamma_i-\sigma(u_iu_j)\gamma_j)^2+\sum\limits_{\substack{\Gamma(u_{m+1} \sim u_j)}}\gamma_j^2-2\sum\limits_{\substack{u_i\in V(\Gamma)}\atop{u_i\neq u_{m+1}}}d_i^-\gamma_i^2}{\gamma_1^2+\gamma_2^2+\cdots+\gamma_m^2}  \\
\end{split}
\end{equation*}
\begin{equation*}
\begin{split}
\alpha_{p+1}&\geqslant \min_{W_{m-p}}~ \max\limits_{\substack{\gamma (\ne 0) \in \mathbb{R}^{m}\\\atop{\gamma \perp W_{m-p}}}}~\left[\frac{\sum\limits_{\substack{\Gamma'(u_i \sim u_j)}}(\gamma_i-\sigma(u_iu_j)\gamma_j)^2-2\sum\limits_{\substack{u_i\in V(\Gamma')}}d_i^-\gamma_i^2}{\gamma_1^2+\gamma_2^2+\cdots+\gamma_m^2}-1\right]=\beta_p-1
\end{split}
\end{equation*}\\
and
\begin{equation*}
\begin{split}
\alpha_{p}&=\max_{\hat{W}_{p-1}}~ \min\limits_{\substack{\hat{\gamma} (\ne 0) \in \mathbb{R}^{m+1}}\atop{\hat{\gamma}\perp \hat{W}_{p-1}}}~\frac{\hat{\gamma}^TN(\Gamma)\hat{\gamma}}{\hat{\gamma}^T\hat{\gamma}}  \\
&=\max_{\hat{W}_{p-1}}~ \min\limits_{\substack{\hat{\gamma} (\ne 0) \in \mathbb{R}^{m+1}}\atop{\hat{\gamma}\perp \hat{W}_{p-1}}}~\frac{\sum\limits_{\substack{\Gamma(u_i \sim u_j)}}(\gamma_i-\sigma(u_iu_j)\gamma_j)^2-2\sum\limits_{\substack{u_i\in V(\Gamma)}}d_i^-\gamma_i^2}{\gamma_1^2+\gamma_2^2+\cdots+\gamma_m^2+\gamma_{m+1}^2}\\
&\leqslant \max_{\hat{W}_{p-1}}~ \min\limits_{\substack{\hat{\gamma} (\ne 0) \in \mathbb{R}^{m+1}\\\hat{\gamma}\perp \hat{W}_{p-1}}\atop{\hat{\gamma} \perp e_{m+1}}}~\frac{\sum\limits_{\substack{\Gamma'(u_i \sim u_j)}}(\gamma_i-\sigma(u_iu_j)\gamma_j)^2+\sum\limits_{\substack{\Gamma(u_{m+1} \sim u_j)}}(\gamma_{m+1}-\sigma(u_{m+1}u_j)\gamma_j)^2-2\sum\limits_{\substack{u_i\in V(\Gamma)}}d_i^-\gamma_i^2}{\gamma_1^2+\gamma_2^2+\cdots+\gamma_m^2+\gamma_{m+1}^2}  \\
&= \max_{W_{p-1}}~ \min\limits_{\substack{\gamma (\ne 0) \in \mathbb{R}^{m}\\\atop{\gamma \perp W_{p-1}}}}~\frac{\sum\limits_{\substack{\Gamma'(u_i \sim u_j)}}(\gamma_i-\sigma(u_iu_j)\gamma_j)^2+\sum\limits_{\substack{\Gamma(u_{m+1} \sim u_j)}}\gamma_j^2-2\sum\limits_{\substack{u_i\in V(\Gamma)}\atop{u_i\neq u_{m+1}}}d_i^-\gamma_i^2}{\gamma_1^2+\gamma_2^2+\cdots+\gamma_m^2}\\
&\leqslant \max_{W_{p-1}}~ \min\limits_{\substack{\gamma (\ne 0) \in \mathbb{R}^{m}\\\atop{\gamma \perp W_{p-1}}}}~\Bigg[\frac{\sum\limits_{\substack{\Gamma'(u_i \sim u_j)}}(\gamma_i-\sigma(u_iu_j)\gamma_j)^2-2\sum\limits_{\substack{u_i\in V(\Gamma')}}d_i^-\gamma_i^2}{\gamma_1^2+\gamma_2^2+\cdots+\gamma_m^2}+1\Bigg]=\beta_p+1
\end{split}
\end{equation*}
Thus, $\alpha_p-1\leqslant \beta_p\leqslant \alpha_{p+1}+1$ for $p=1,2,\cdots,m$.
\end{proof}
\begin{cor}{\textnormal{Let $\Gamma=(G,\sigma)$ be a signed graph of order $m+1$ and $\Gamma'=\Gamma-v$ for a vertex $v\in V(\Gamma)$ such that $d_v^-=0$. If $(\alpha_1,\alpha_2,\cdots,\alpha_{m+1})$ and $(\beta_1,\beta_2,\cdots,\beta_m)$ are the ordered net-Laplacian spectrum of $\Gamma$ and $\Gamma'$ respectively then
\[
\alpha_p-1\leqslant \beta_p \leqslant \alpha_{p+1} \text{ for } p=1,2,\cdots,m
\]}}
\end{cor}
\begin{proof} Proof follows from Theorem \ref{thm 3.4} since $\sum\limits_{\substack{u_i\in V(\Gamma)}\atop{u_i \neq v}}d_i^-\gamma_i^2=\sum\limits_{\substack{u_i\in V(\Gamma')}}d_i^-\gamma_i^2$.
\end{proof}
\begin{cor}{\textnormal{Let $\Gamma=(G,\sigma)$ be a signed graph of order $m+1$ and $\Gamma'=\Gamma-v$ for a vertex $v\in V(\Gamma)$ such that $d_v^+=0$. If $(\alpha_1,\alpha_2,\cdots,\alpha_{m+1})$ and $(\beta_1,\beta_2,\cdots,\beta_m)$ are the ordered net-Laplacian spectrum of $\Gamma$ and $\Gamma'$ respectively then
\[
\alpha_p\leqslant \beta_p \leqslant \alpha_{p+1}+1 \text{ for } p=1,2,\cdots,m
\]}}
\end{cor}
\begin{proof} Proof follows from Theorem \ref{thm 3.4} since $\sum\limits_{\substack{\Gamma(v \sim u_j)}}\gamma_j^2-2\sum\limits_{\substack{u_i\in V(\Gamma)}\atop{u_i \neq v}}d_i^-\gamma_i^2=-2\sum\limits_{\substack{u_i\in V(\Gamma')}}d_i^-\gamma_i^2$.
\end{proof}
\begin{cor}{\textnormal{Let $\Gamma=(G,\sigma)$ be a $(m,k)$ co-regular signed graph of order $m+1$ where $d_i^-=s$ for all $i=1,2,\cdots,m+1$ and $\Gamma'=\Gamma-v$ for a vertex $v\in V(\Gamma)$. If $(\alpha_1,\alpha_2,\cdots,\alpha_{m+1})$ and $(\beta_1,\beta_2,\cdots,\beta_m)$ are the ordered net-Laplacian spectrum of $\Gamma$ and $\Gamma'$ respectively then
\[
\beta_p+2s\leqslant \alpha_p\leqslant \mu_p+(1-2s)\leqslant \alpha_{p+1}\leqslant\beta_{p+1}+2s \text{ for } p=1,2,\cdots,m
\]
where $(\mu_1,\mu_2,\cdots,\mu_m)$ is the ordered Laplacian spectrum of $\Gamma'$.}}
\end{cor}
\begin{proof} Let $v=u_{m+1}$. So,  ${\sum\limits_{\substack{\Gamma(u_{m+1} \sim u_j)}}(\gamma_{m+1}-\sigma(u_{m+1}u_j)\gamma_j)^2}={\displaystyle \sum_{j=1}^{m}(\gamma_{m+1}-\sigma(u_{m+1}u_j)\gamma_j)^2}$. Thus
\begin{equation*}
\begin{split}
\alpha_{p+1}&\geqslant \min_{W_{m-p}}~ \max\limits_{\substack{\gamma (\ne 0) \in \mathbb{R}^{m}\\\atop{\gamma\perp W_{m-p}}}}~\frac{\sum\limits_{\substack{\Gamma'(u_i \sim u_j)}}(\gamma_i-\sigma(u_iu_j)\gamma_j)^2+\displaystyle \sum_{j=1}^{m}\gamma_j^2-2s\sum_{i=1}^{m}\gamma_i^2}{\gamma_1^2+\gamma_2^2+\cdots+\gamma_m^2}  \\ \\
&= \min_{W_{m-p}} ~\max\limits_{\substack{\gamma (\ne 0) \in \mathbb{R}^{m}\\\atop{\gamma \perp W_{m-p}}}}~\left[\frac{\sum\limits_{\substack{\Gamma'(u_i \sim u_j)}}(\gamma_i-\sigma(u_iu_j)\gamma_j)^2}{\gamma_1^2+\gamma_2^2+\cdots+\gamma_m^2}+(1-2s)\right]=\mu_p+(1-2s)
\end{split}
\end{equation*}
and
\begin{equation*}
\begin{split}
\alpha_p&\leqslant \max_{W_{p-1}}~ \min\limits_{\substack{\gamma (\ne 0) \in \mathbb{R}^{m}\\\atop{\gamma\perp W_{p-1}}}}~\frac{\sum\limits_{\substack{\Gamma'(u_i \sim u_j)}}(\gamma_i-\sigma(u_iu_j)\gamma_j)^2+\displaystyle \sum_{j=1}^{m}\gamma_j^2-2s\sum_{i=1}^{m}\gamma_i^2}{\gamma_1^2+\gamma_2^2+\cdots+\gamma_m^2}  \\ \\
&= \max_{W_{p-1}} ~\min\limits_{\substack{\gamma (\ne 0) \in \mathbb{R}^{m}\\\atop{\gamma \perp W_{p-1}}}}~\left[\frac{\sum\limits_{\substack{\Gamma'(u_i \sim u_j)}}(\gamma_i-\sigma(u_iu_j)\gamma_j)^2}{\gamma_1^2+\gamma_2^2+\cdots+\gamma_m^2}+(1-2s)\right]=\mu_p+(1-2s)
\end{split}
\end{equation*}
Combining the above two inequalities and using Lemma \ref{lem 3.1} will give the required result.
\end{proof}
\section{Normalized net-Laplacian of a signed graph}
\label{sec4}
For a signed graph $\Gamma$ with vertex set $\{u_1,u_2,\cdots,u_m\}$, a net-Laplacian matrix is normalized as,
\[
\Bar{N}(\Gamma)=D(\Gamma)^{-\frac{1}{2}}N(\Gamma)D(\Gamma)^{-\frac{1}{2}}
\]
where,
\[
(D(\Gamma)^{-\frac{1}{2}})_{ij}=\begin{cases}
    0 &\text{ if } i\neq j\\
    \frac{1}{\sqrt{d_{u_{_i}}}} &\text{ if } i=j \text{ and } d_{u_{_i}}\neq 0\\
    0 &\text{ if } i=j \text{ and } d_{u_{_i}}= 0
\end{cases}
\]
If $V(\Gamma)=\{u_1,u_2,\cdots,u_m\}$ and $\gamma=\begin{bmatrix}
    \gamma_1\\
    \gamma_2\\
    \vdots\\
    \gamma_m
\end{bmatrix}\in \mathbb{R}^m$ then,
\begin{equation*}
\frac{\gamma^TN(\Gamma)\gamma}{\gamma^T\gamma}=\frac{\sum\limits_{\substack{\Gamma(v_i\sim v_j)}}(\gamma_i-\sigma(v_iv_j)\gamma_j)^2-2\sum\limits_{\substack{v_i\in V(\Gamma)}}d_i^-\gamma_i^2}{\displaystyle \sum_{i=1}^{n}\gamma_i^2}
\end{equation*}
Let $y=D(\Gamma)^{\frac{1}{2}}\gamma$ then,
\begin{align}
\frac{y^TN(\Gamma)y}{y^Ty}&=\frac{\left(D(\Gamma)^{\frac{1}{2}}\gamma\right)^T\left(D(\Gamma)^{-\frac{1}{2}}N(\Gamma)D(\Gamma)^{-\frac{1}{2}}\right)\left(D(\Gamma)^{-\frac{1}{2}}\gamma\right)}{\left(D(\Gamma)^{\frac{1}{2}}\gamma\right)^T\left(D(\Gamma)^{\frac{1}{2}}\gamma\right)} \nonumber \\
&=\frac{\gamma^TN(\Gamma)\gamma}{\gamma^TD\gamma} \nonumber\\
&=\frac{\sum\limits_{\substack{\Gamma(u_i\sim u_j)}}(\gamma_i-\sigma(u_iu_j)\gamma_j)^2-2\sum\limits_{\substack{u_i\in V(\Gamma)}}d_i^-\gamma_i^2}{\displaystyle \sum_{i=1}^{m}\gamma_i^2d_i}
\end{align}
If $(\mu_1,\mu_2,\cdots,\mu_m)$ is the ordered normalized net-Laplacian spectrum of a signed graph $\Gamma=(G,\sigma)$ with vertex set $\{u_1,u_2,\cdots,u_m\}$ then from~\cite{ref4}, it follows that \[
\mu_p\leqslant \max_{\gamma \neq 0}~\frac{\sum\limits_{\substack{\Gamma(u_i\sim u_j)}}(\gamma_i-\sigma(u_iu_j)\gamma_j)^2-2\sum\limits_{\substack{u_i\in V(\Gamma)}}d_i^-\gamma_i^2}{\displaystyle \sum_{i=1}^{m}\gamma_i^2d_i}\leqslant \max_{\gamma \neq 0}~\frac{\sum\limits_{\substack{\Gamma(u_i\sim u_j)}}(\gamma_i-\sigma(u_iu_j)\gamma_j)^2}{\displaystyle \sum_{i=1}^{m}\gamma_i^2d_i} \leqslant 2
\]
Equality holds for $p=m$ when $\gamma_i=-\sigma(u_iu_j)\gamma_j$ for every edge $u_iu_j$ in $G$ and $d_i^-=0$ for all $i=1,2,\cdots,m$. As $\gamma \neq 0$, $G$ has a bipartite connected component and $\sigma=+$. Also
\[
\mu_p\geqslant \min_{\gamma \neq 0}\frac{\sum\limits_{\substack{\Gamma(u_i\sim u_j)}}(\gamma_i-\sigma(u_iu_j)\gamma_j)^2-2\sum\limits_{\substack{u_i\in V(\Gamma)}}d_i^-\gamma_i^2}{\displaystyle \sum_{i=1}^{m}\gamma_i^2d_i}\geqslant \min_{\gamma \neq 0}\frac{-2\sum\limits_{\substack{u_i\in V(\Gamma)}}d_i^-\gamma_i^2}{\displaystyle \sum_{i=1}^{m}\gamma_i^2d_i}\geqslant -2
\]
and equality holds for $p=1$ when $\gamma_i=\sigma(u_iu_j)\gamma_j$ for every edge $u_iu_j$ in $G$ and $d_i^-=d_i$ for all $i=1,2,\cdots,m$. As $\gamma \neq 0$, $G$ has a bipartite connected component and $\sigma=-$.
\begin{therm}\textnormal{Let $\Gamma$ be a signed graph without isolated vertices of order $m$ and let $\Gamma'=\Gamma-e$ be the signed graph obtained by removing an edge $e=uv$ such that $\sigma(e)=-$. If $(\alpha_1,\alpha_2,\cdots,\alpha_m)$ and $(\beta_1,\beta_2,\cdots,\beta_m)$ are the ordered spectrum of $\Bar{N}(\Gamma)$ and $\Bar{N}(\Gamma')$ respectively, then}
    \[
    \alpha_p\leqslant \beta_p\leqslant \alpha_{p+2}~ ; ~\text{ for } p=1,2,\cdots,m
    \]
\end{therm}
\begin{proof}
    Taking $v=u_1$ and $u=u_2$ and proceeding as in Theorem \ref{theorem 3.2}, the eigenvalues $\beta_p$ of $\Bar{N}(\Gamma')$ can be expressed as
    \[
    \beta_p=\min_{W_{m-p}} ~\max\limits_{\substack{\gamma (\ne 0) \in \mathbb{R}^{m}}\atop{\gamma\perp W_{m-p}}}~\frac{\sum\limits_{\substack{\Gamma(u_i \sim u_j)}}(\gamma_i-\sigma(u_iu_j)\gamma_j)^2-2\sum\limits_{\substack{u_i\in V(\Gamma)}}d_i^-\gamma_i^2+(\gamma_1-\gamma_2)^2}{\sum\limits_{\substack{\Gamma:i=1}}^{m}d_i\gamma_i^2-\gamma_1^2-\gamma_2^2}
    \]
    Similar to Theorem \ref{theorem 3.2}, the lower bound $\alpha_p$ follows and the upper bound $\alpha_{p+2}$ follows taking $\gamma \perp e_1,e_2$ in the max-min statement of Courant-Fisher Theorem \ref{v}.
\end{proof}
\begin{therm}\textnormal{Let $\Gamma$ be a signed graph without isolated vertices of order $m$ and let $\Gamma'=\Gamma-e$ be the signed graph obtained by removing an edge $e=uv$ such that $\sigma(e)=+$. If $(\alpha_1,\alpha_2,\cdots,\alpha_m)$ and $(\beta_1,\beta_2,\cdots,\beta_m)$ are the ordered spectrum of $\Bar{N}(\Gamma)$ and $\Bar{N}(\Gamma')$ respectively, then}
    \[
    \alpha_{p-1}\leqslant \beta_p\leqslant \alpha_{p+1}~ ; ~\text{ for } p=1,2,\cdots,m
    \]
\end{therm}
\begin{proof}
    Taking $v=u_1$ and $u=u_2$ and proceeding as in Theorem \ref{theorem 3.2}, the eigenvalues $\beta_p$ of $\Bar{N}(\Gamma')$ can be expressed as
    \[
    \beta_p=\min_{W_{m-p}} ~\max\limits_{\substack{\gamma (\ne 0) \in \mathbb{R}^{m}}\atop{\gamma\perp W_{m-p}}}~\frac{\sum\limits_{\substack{\Gamma(u_i \sim u_j)}}(\gamma_i-\sigma(u_iu_j)\gamma_j)^2-2\sum\limits_{\substack{u_i\in V(\Gamma)}}d_i^-\gamma_i^2-(\gamma_1-\gamma_2)^2}{\sum\limits_{\substack{\Gamma:i=1}}^{m}d_i\gamma_i^2-\gamma_1^2-\gamma_2^2}
    \]
    Similar to Theorem 8 in \cite{ref1}, taking $\gamma_1=\gamma_2$ gives the lower bound $\alpha_{p-1}$ and the upper bound $\alpha_{p+1}$ follows taking $\gamma_1=-\gamma_2$ in the max-min statement of Courant-Fisher Theorem \ref{v}.
\end{proof}
\subsection{Vertex contraction}
Consider a signed graph $\Gamma$, and let $t \in V(\Gamma)$. The open neighborhood of the vertex $t$ is defined as the set of vertices adjacent to $t$ in $\Gamma$ that is
\[
\mathcal{N}(t)=\{w \in V(\Gamma);~~ tw \in E(\Gamma)\}
\]
and $\mathcal{N}[t]=\mathcal{N}(t)\cup \{t\}$ is a closed neighborhood of $t$ in $\Gamma$.\\

For a signed graph $\Gamma=(G,\sigma)$ with vertices $a$ and $b$ such that $\sigma(ta)=\sigma(tb)$ for all $t \in \mathcal{N}(a) \cap \mathcal{N}(b)$, Atay and Tuncel in~\cite{ref1} defined an allowable contraction $\Gamma/\{a,b\}$ on $a$ and $b$, as a signed graph obtained from $\Gamma$ by deleting the vertices $a$ and $b$ and adding a new vertex $(ab)$ such that $\mathcal{N}(ab)=\mathcal{N}(a) \cup \mathcal{N}(b)$ and
\[
\sigma(t(ab))=\begin{cases}
    \sigma(ta) & \text{ if } t\in \mathcal{N}(a) \text{ and } t \notin \mathcal{N}(b)\\
    \sigma(tb) & \text{ if } t\in \mathcal{N}(b) \text{ and } t \notin \mathcal{N}(a)\\
    \sigma(ta)=\sigma(tb) & \text{ if } t \in \mathcal{N}(a) \cap \mathcal{N}(b)
\end{cases}
\]
In particular, if $\mathcal{N}(a) \cup \mathcal{N}[b]=\emptyset$ then $\Gamma/\{a,b\}$ is an allowable contraction.
\begin{therm}
\textnormal{Let $\Gamma$ be a signed graph and $a,b \in V(\Gamma)$ such that $\mathcal{N}(a) \cap\mathcal{N}(b)=\emptyset$. If $(\alpha_1,\alpha_2,\cdots,\alpha_{m+1})$ and $(\beta_1,\beta_2,\cdots,\beta_m)$ are the ordered spectrum of $\Bar{N}(\Gamma)$ and $\Bar{N}(\Gamma')$ respectively where $\Gamma'=\Gamma/\{a,b\}$, then for $p=1,2,\cdots,m$
\[
\alpha_{p-1}\leqslant\beta_p\leqslant\alpha_{p+1}
\]
with the convention that $\alpha_0=-2$}
\end{therm}
\begin{proof}
    Let $a=u_1$, $b=u_2$ and $k$ denote an index set such that $k\in K$ if and only if $u_k \in N(u_1)$. As, $\mathcal{N}(u_1)\cap \mathcal{N}(u_2)=\emptyset$, $\Gamma/\{a,b\}$ can be obtained from $\Gamma$ by removing edges $u_1u_k$ and adding edges $u_2u_k$ such that $\sigma(u_2u_k)=\sigma(u_1u_k)$ for all $k\in K$. Using Courant-Fisher Theorem \ref{v}, the eigenvalues $\beta_p$ of $\mathcal{N}(\Gamma')$ is given by
    \begin{equation*}
        \begin{split}
            \beta_p&=\min_{W_{m-p}} ~\max\limits_{\substack{\gamma \ne 0 }\atop{\gamma\perp W_{m-p}}}~\frac{1}{\sum\limits_{\substack{u_i\in V(\Gamma)}}d_i\gamma_i^2-\gamma_1^2-\gamma_2^2}\Bigg[\sum\limits_{\substack{\Gamma(u_i \sim u_j)}}(\gamma_i-\sigma(u_iu_j)\gamma_j)^2+\sum\limits_{\substack{k\in K}}\Big[(\gamma_2-\sigma(u_2u_k)\gamma_k)^2\\&~~~~~~~~~~~~~~~~~~~~~~~~~~~~~~~~~~~~~~~~~~~~~-(\gamma_1-\sigma(u_1u_k)\gamma_k)^2\Big]-2\Big(\sum\limits_{\substack{u_i\in V(\Gamma)}}d_i^-\gamma_i^2-d_1^-\gamma_1^2+d_1^-\gamma_2^2\Big) \Bigg]
        \end{split}
    \end{equation*}
    Similar to Theorem 2.7 in~\cite{ref3}, taking $\gamma_1=\gamma_2$ on both min-max and max-min statement of the Courant-Fisher Theorem \ref{v} gives the lower bound $\alpha_{p-1}$ and upper bound $\alpha_{p+1}$ respectively.
\end{proof}
 \section*{Acknowledgements}
We would like to acknowledge National Institute of Technology Sikkim for giving doctoral fellowship to Satyam Guragain.

\end{document}